\documentclass[12pt]{article}

\usepackage{amsthm,amsfonts,amsbsy,amssymb,amsmath}
\usepackage[all]{xy}
\usepackage{xr-hyper}
\usepackage[colorlinks=true]{hyperref}
\usepackage{fullpage}
\usepackage{verbatim}
\usepackage[dvips]{lscape}
\usepackage{hyperref} 

\begin{document}
\title{On Volumes of Arithmetic Line Bundles II}
\author{Xinyi Yuan
\footnote{The author is fully supported by a research fellowship of the Clay Mathematics Institute.}
}
\maketitle

\theoremstyle{plain}
\newtheorem{thm}{Theorem}[section]
\newtheorem*{conj}{Conjecture}
\newtheorem*{notation}{Notation}
\newtheorem{cor}[thm]{Corollary}
\newtheorem*{corr}{Corollary}
\newtheorem{lem}[thm]{Lemma}
\newtheorem{pro}[thm]{Proposition}
\newtheorem{definition}[thm]{Definition}
\newtheorem*{thmm}{Theorem}

\theoremstyle{remark} \newtheorem*{remark}{Remark}
\theoremstyle{remark} \newtheorem*{example}{Example}

\newcommand{\RR}{\mathbb{R}}
\newcommand{\QQ}{\mathbb{Q}}
\newcommand{\CC}{\mathbb{C}}
\newcommand{\ZZ}{\mathbb{Z}}
\newcommand{\NN}{\mathbb{N}}
\newcommand{\FF}{\mathbb{F}}
\newcommand{\PP}{\mathbb{P}} 
\renewcommand{\AA}{\mathbb{A}} 

\newcommand{\fp}{\mathbb{F}_p}
\newcommand{\fwp}{\mathbb{F}_{\wp}}

\newcommand{\rank}{\mathrm{rank}}         
\newcommand{\vol}{\mathrm{vol}}           
\newcommand{\volchi}{\mathrm{vol}_{\chi}} 
\newcommand{\divv}{\mathrm{div}}          
\newcommand{\Spec}{\mathrm{Spec}}         
\newcommand{\ord}{\mathrm{ord}}           

\newcommand{\lb}{\mathcal{L}}             
\newcommand{\mb}{\mathcal{M}}
\newcommand{\nb}{\mathcal{N}}
\newcommand{\fb}{\mathcal{F}}
\newcommand{\eb}{\mathcal{E}}
\newcommand{\tb}{\mathcal{T}}
\newcommand{\ob}{\mathcal{O}}
\newcommand{\ib}{\mathcal{I}}
\newcommand{\jb}{\mathcal{J}}
\newcommand{\ab}{\mathcal{A}}
\newcommand{\bb}{\mathcal{B}} 

\newcommand{\hhat}{\hat h^0}
\newcommand{\Hhat}{\widehat H^0}

\newcommand{\xb}{\mathcal{X}}             
\newcommand{\cx}{\mathcal{X}}             

\newcommand{\lbb}{\overline{\mathcal{L}}}             
\newcommand{\mbb}{\overline{\mathcal{M}}}
\newcommand{\nbb}{\overline{\mathcal{N}}}
\newcommand{\ebb}{\overline{\mathcal{E}}}
\newcommand{\tbb}{\overline{\mathcal{T}}}
\newcommand{\obb}{\overline{\mathcal{O}}}
\newcommand{\abb}{\overline{\mathcal{A}}}
\newcommand{\bbb}{\overline{\mathcal{B}}}

\newcommand{\lbt}{\widetilde{\mathcal{L}}}
\newcommand{\mbt}{\widetilde{\mathcal{M}}}

\newcommand{\chern}{\hat{c}_1}           
\newcommand{\height}{h_{\lbb}}           
\newcommand{\cheight}{\hat{h}_{\lb}}     

\newcommand{\lnorm}[1]{\|#1\|_{L^2}}     
\newcommand{\supnorm}[1]{\|#1\|_{\mathrm{sup}}}    
\newcommand{\chil}{\chi_{_{L^2}}}
\newcommand{\chisup}{\chi_{\sup}}

\newcommand{\amp}{\widehat{\mathrm{Amp}}(X)}   
\newcommand{\nef}{\widehat{\mathrm{Nef}}(X)}   
\newcommand{\eff}{\widehat{\mathrm{Eff}}(X)}   
\newcommand{\pic}{\widehat{\mathrm{Pic}}(X)}   
\newcommand{\bigx}{\widehat{\mathrm{Big}}(X)}

\newcommand{\body}{\Delta(L)}
\newcommand{\bodyin}{\Delta(L)^\circ}

\tableofcontents

\section{Introduction}
This paper uses convex bodies to study line bundles in the setting of Arakelov theory. The treatment is parallel to \cite{Yu2}, but the content is independent.

The method of constructing a convex body in Euclidean space, now called ``Okounkov body'', from a given algebraic linear series was due to Okounkov \cite{Ok1, Ok2}, and was explored systematically by Kaveh--Khovanskii \cite{KK} and Lazarsfeld--Musta\c t\v a \cite{LM}. Many important results of algebraic geometry can be derived from convex geometry through the bridge that the volume of the convex body gives the volume of the linear series.

Let $K$ be a number field, $\cx$ be an arithmetic variety of relative dimension $d$ over $O_K$, and $\lbb$ be a hermitian line bundle over $\cx$. There are two important arithmetic invariants $\hhat(\lbb)$ and $\chi(\lbb)$. 
Their growth under tensor powers are measured respectively by $\vol(\lbb)$ and $\volchi(\lbb)$. In \cite{Yu2}, we have introduced the Okounkov body $\Delta(\lbb) \subset \RR^{d+1}$ of $\lbb$, whose volume computes $\vol(\lbb)$. It is a natural arithmetic analogue of the construction in \cite{LM}. 

In the current paper, we use the usual Okounkov body $\Delta(\lb_K)\subset \RR^{d}$ of the generic fibre $\lb_K$ viewed as a line bundle on the projective variety $\cx_K$. Then we introduce the Chebyshev transform $c[\lbb]$ of $\lbb$, which is a convex function on $\Delta(\lb_K)$. We show that the Lebesgue integral of $c[\lbb]$ on $\Delta(\lb_K)$ gives $\volchi(\lbb)$ under some boundedness condition. For example, it is true in the case of toric varieties. We conjecture that it is true in the general case that the generic fibre $\lb_K$ is big.

The construction of $c[\lbb]$ is inspired by the work of Nystr\"om \cite{Ny}, and can be viewed as the global case of Nystr\"om's work. For each place $v$ of $K$, we introduce a local Chebyshev transform $c_v[\lbb]$, which is a real-valued convex function on $\Delta(\lb_K)$ depending on the $v$-adic metric of $\lb_K$ induced by the model $\lbb$. The archimedean case is exactly Nystr\"om's construction, and the non-archimedean case is analogous. Then the global Chebyshev transform is defined by $c[\lbb]=\sum_v c_v[\lbb].$

\subsection{Algebraic case}

We first review the construction of \cite{LM}. 
Instead of flag of subvarieties, we will use local coordinates as in \cite{Ny}. It actually gives a generic infinitesimal flag in the sense of \cite{LM}.

Let $X$ be a projective variety of dimension $d$ over a base field $K$.
Assume that there is a rational point $x_0\in X(K)$ in the smooth locus of $X$.
Let $t=(t_1,\cdots, t_d)$ be a system of parameters at $x_0$. In another word, 
$t_1,\cdots, t_d$ are elements of the maximal ideal $m_{x_0}\subset O_{X,x_0}$ 
whose images in $m_{x_0}/m_{x_0}^2$ form a $K$-basis of $m_{x_0}/m_{x_0}^2$.

Let $L$ be a line bundle over $X$. Fix a local section $s_0$ of $L$ around $x_0$ which does not vanish at $x_0$. It gives a trivialization $mL\subset O_{X,x_0}$ for any integer $m$. Here $mL$ means $L^{\otimes m}$, and we always write line bundles additively in this paper. In fact, 
for any local section $s$ of $mL$ at $x_0$, the quotient $s_0^{-m}s$ is a rational function on $X$ and we may identify $s$ with $s_0^{-m}s$. Consider the power series expansion 
$$
s=\sum_{\alpha\in \NN^d} a_\alpha t^\alpha.
$$
Here we use the convention $t^\alpha:=t_1^{\alpha_1}\cdots t_d^{\alpha_d}$ for any 
$\alpha=(\alpha_1, \cdots, \alpha_d)\in \NN^d$.
The smallest $\alpha\in \NN^d$ with respect to the lexicographic order of $\ZZ^d$ such that $a_\alpha\neq 0$, denoted by $\nu_t(s)$, is called the \emph{valuation of} $s$.

It follows that we have a map 
$$
\nu_t: H^0(X,mL)-\{0\} \rightarrow \NN^d.
$$
The \emph{Okounkov body} $\Delta_t(L)$ is defined to be the closure of 
$$\Lambda_t(L)=\bigcup_{m\geq 1}\frac{1}{m} \nu_t(H^0(X,mL))$$
in $\RR^d$. Here $\nu_t(H^0(X,mL))$ stands for $\nu_t\left(H^0(X,mL)-\{0\}\right)$ throughout this paper.

As shown by \cite{LM}, the Okounkov body is convex and bounded and satisfies
$$
\vol(\Delta_{t}(L)) =\frac{1}{d!}\vol(L). $$
Here we recall that the volume of $L$ is defined by
\begin{eqnarray*}
\vol(L)=\lim_{m\rightarrow \infty}  \frac{\dim_K H^0(X,mL)}{m^d/d!}.
\end{eqnarray*}
The existence of the above limit is an easy consequence of properties of the Okounkov body, and was implied by Fujita's approximation theorem in history. 

To prepare for the arithmetic case, we introduce one more notation. For any $\displaystyle \alpha\in \nu_t(H^0(X,L))$, define
$$H^0(X,L)(\alpha)= \{s \in H^0(X,L): s=t^{\alpha}+\mathrm{\ higher\ order\ terms} \}.$$ 
In another word, $H^0(X,L)(\alpha)$ consists of sections with valuation $\alpha$ and leading coefficient equal to 1.
Then $H^0(X,L)(\alpha)$ gives a way to normalize the sections.
Any $s\in H^0(X,L)$ is uniquely written in the form $\sum_{\alpha\in \nu_t(H^0(X,L))}b_{\alpha} s_{\alpha}$ with $b_{\alpha}\in K$ and $s_{\alpha}\in H^0(X,L)(\alpha)$. 

The valuation map $\nu_t$ and the Okounkov body $\Delta_t(L)$ depend on the choice of $x_0$ and $t$, but not on the choice of $s_0$. The set $H^0(X,L)(\alpha)$ will be multiplied by $s'_0(x_0)/s_0(x_0)\in K^\times$ if we change $s_0$ to another $s_0'$.

\subsection{Chebyshev transform}
To define the Chebyshev transform, it is more convenient to work on adelic metrized line bundles in the sense of Zhang \cite{Zh}. We briefly recall the definition. For more details, we refer to \cite{Zh} and \S\ref{preliminary} of the current paper.

Let $X$ be a projective variety of dimension $d$ over a number field $K$, and $L$ be a line bundle over $X$. For any place $v$ of $K$, by a \emph{$v$-adic metric} on $L$ we mean an assignment of a Galois invariant $\CC_v$-norm to the fibre $L_{\CC_v}(x)$ at any point $x\in X(\CC_v)$. Here $\CC_v$ denotes the completion of the algebraic closure of $K_v$. It induces the supremum norms on $H^0(X_{\CC_v}, L_{\CC_v})$ given by 
$$
\|s\|_{v,\sup}=\sup_{x\in X(\CC_v)}  \|s(x)\|_v.
$$
An \emph{adelic metric} on $L$ is a collection $\{\|\cdot\|_v\}_v$ of continuous $v$-adic metrics on $L$ over all places $v$ of $K$ satisfying certain extra conditions. We usually write $\bar L=(L, \{\|\cdot\|_v\}_v)$ and call it an \emph{adelic metrized line bundle}. 

Let $\bar L=(L, \{\|\cdot\|_v\}_v)$ be an adelic metrized line bundle on $X$ with $L$ big. As in the algebraic case, take a rational regular point $x_0\in X(K)$ which exists by enlarging $K$, take a local coordinate $t=(t_1,\cdots, t_d)$ at $x_0$, and take a base local section $s_0$ that induces trivialization of $L$. 
The Okounkov body $\Delta(L)=\Delta_{t}(L)$ is induced by the valuation map $\nu=\nu_{t}$ on the global sections.

Fix a place $v$. The local Chebyshev transform $c_v[\bar L]$ of $\bar L$ will be a real-valued function on $\Delta(L)$ constructed from $\|\cdot\|_v$ by the method of \cite{Ny}.
For any $\alpha\in \nu(H^0(X,L))$,
define the \emph{local discrete Chebyshev transform}
$$
F_v[\bar L](\alpha):= \inf_{s\in H^0(X_{K_v},L_{K_v})(\alpha)}  \log \|s\|_{v,\sup}.
$$
Here
$$H^0(X_{K_v},L_{K_v})(\alpha)= \{s \in H^0(X_{K_v},L_{K_v}): s=t^{\alpha}+\mathrm{\ higher\ order\ terms} \}$$ 
is introduced above.

Let $\{m_k\}_k$ be a sequence of positive integers tending to infinity, 
and let $\alpha_k\in \frac{1}{m_k}\nu(H^0(X,m_kL))$ be a sequence convergent to some
$\alpha\in \Delta(L)$ under the Euclidean topology. 
Then $c_v[\bar L](\alpha)$ is defined to be the limit
$$
\lim_{k\rightarrow\infty} \frac{1}{m_k} F_v[m_k\bar L](m_k\alpha_k).
$$

\begin{pro}\label{definition}
For any $\alpha$ in the interior $\Delta(L)^\circ$ of $\Delta(L)$, the above limit exists and does not depend on the sequence $\{(m_k,\alpha_k)\}$. Thus 
$$
c_v[\bar L](\alpha):= \lim_{k\rightarrow\infty} \frac{1}{m_k} F_v[m_k\bar L](m_k\alpha_k).
$$
is a well-defined function on $\Delta(L)^\circ$. 
Furthermore, it is convex and continuous.
\end{pro}

If $v$ is archimedean, the result is exactly Theorem 1.1 and Theorem 1.2 of \cite{Ny}. The non-archimedean case is very similar. Note that our definition of $c_v$ differs from that in \cite{Ny} by a factor $2$, though a more natural definition should use a factor $-2$.

Go back to the globally metrized line bundle $\bar L$. We simply define
$$
c[\bar L]:=\sum_v c_v[\bar L].
$$
At every point of $\Delta(L)^\circ$, the sum has only finitely many non-zero terms. Then we check that $c[\bar L]$ defines a convex and continuous function on $\Delta(L)^\circ$. By the product formula, $c[\bar L]$ is independent of the choice of $s_0$.

\subsection{Integration and volume}

Let $(X,\bar L)$ be as above. The main result of this paper is as follows:

\begin{thm}\label{integration}
Assume that $L$ is big. Then the following are true:
\begin{itemize}
\item[(1)]
The function $c[\bar L]$ is integrable on $\Delta(L)$, and
$$
\int_{\Delta(L)} c[\bar L](\alpha)d\alpha
\leq -\frac{1}{d!}\vol_{\chi}(\bar L).
$$
\item[(2)]
If the function $\displaystyle \frac 1m \sum_v F_v[m\bar L]$ is uniformly bounded for all $m\geq 1$, then 
$$
\int_{\Delta(L)} c[\bar L](\alpha)d\alpha
=-\frac{1}{d!}\vol_{\chi}(\bar L).
$$
\end{itemize}
\end{thm}

Now we explain the definition of the adelic volume $\vol_{\chi}(\bar L)$.
We first denote
\begin{eqnarray*}
\chi(\bar L) = \log \frac{\vol(B(\bar L))}{\vol(H^0(X,L)_{\AA_K}/H^0(X,L))}.
\end{eqnarray*}
Here $\AA_K=\prod_vK_v$ is the adele ring of $K$ and $H^0(X,L)_{\AA_K}$ is the adelization of $H^0(X,L)$.
The unit ball
$$
B(\bar L)=\prod_v B_v(\bar L)
$$
is a product of local unit balls
$$B_v(\bar L)=\{s\in H^0(X,L)_{K_v}: \|s\|_{v,\sup}\leq 1 \}.$$
Note that both $B(\bar L)$ and $H^0(X,L)_{\AA_K}/H^0(X,L)$ are compact, and the quotient of their volumes does
not depend on the Haar measure on $H^0(X,L)_{\AA_K}$.
Now we can define
\begin{eqnarray*}
\vol_{\chi}(\bar L)=\limsup_{m\rightarrow \infty}  \frac{\chi(m\bar L)}{m^{d+1}/(d+1)!}.
\end{eqnarray*}

Let $\cx$ be an arithmetic variety and $\lbb=(\lb, \|\cdot\|)$ be a hermitian line bundle over $\cx$.
Then an adelic metric $\{\|\cdot\|_v\}_v$ is induced on the line bundle $\lb_K$ over $\cx_K$. 
By the data $(\cx_K, \overline\lb_K=(\lb_K, \{\|\cdot\|_v\}_v) )$, we have the Chebyshev transform 
$c[\overline\lb_K]$ on $\Delta(\lb_K)$ depending on the choice of the local coordinate. 
In this case $\volchi(\overline\lb_K)$ is equal to 
\begin{eqnarray*}
\vol_{\chi}(\lbb)=\limsup_{m\rightarrow \infty}  \frac{\chi(m\lbb)}{m^{d+1}/(d+1)!}.
\end{eqnarray*}
Here the usual arithmetic Euler characteristic is defined by
\begin{eqnarray*}
\chi(\lbb) = \log \frac{\vol(B(\lbb))}{\vol(H^0(\cx,\lb)_{\RR}/H^0(\cx,\lb))}.
\end{eqnarray*}
with unit ball
$$B(\lbb)=\{ s\in H^0(\cx,\lb)_{\RR}: \supnorm{s}\leq 1  \}$$
is the corresponding unit ball. 

The ``limsup'' defining $\vol_{\chi}(\lbb)$ is proved to be a limit by Chen \cite{Ch} using Harder-Narasimhan filtrations.
In the case that Theorem \ref{integration} (2) is true, we can recover Chen's result by our treatment.

\subsection{Connection with previous results}
Let $(X,\bar L)$ be as above. Fix an archimedean place $v_0$. Change the metric of $\bar L$ at $v_0$, and denote the new metrized line bundle by $\bar L'$. Assume that Theorem \ref{integration} (2) is true for both line bundles. Then the difference of these two formulas agrees with Theorem 1.3 of \cite{Ny}.

We can also obtain a log-concavity result as in \cite{LM} and \cite{Yu2}. 
Assume that $c[\bar L]\leq 0$ and that Theorem \ref{integration} (2) is true for $\bar L$. 
From the Okounkov body $\Delta_t(L)\subset \RR^d$, we have the graph
$$
\{(\alpha, y)\in \RR^d\times \RR: \alpha\in \Delta_t(L)^\circ, 
\ 0\leq y\leq -c[\bar L](\alpha)\}.
$$
Denote $\widetilde\Delta_t(\bar L)$ to be the closure of the graph in $\RR^{d+1}$.
Then it gives a convex body in $\RR^{d+1}$ since $c[\bar L]$ is convex.
The volume of $\widetilde\Delta_t(\bar L)$ gives the arithmetic volume 
$\vol_{\chi}(\bar L)$. This convex body is very similar to the one constructed in \cite{Yu2}. 

Assume further that we have a metrized line bundle $\bar M$ on $X$ such that both $\bar M$ and $\bar L+\bar M$ satisfy similar conditions as above. Then it is easy to verify that 
$$\widetilde\Delta_t(\bar L)+ \widetilde\Delta_t(\bar M) \subset \widetilde\Delta_t(\bar L+\bar M).$$
As in \cite{Yu2}, the Brunn-Minkowski inequality gives 
$$
\vol_{\chi}(\bar L+\bar M)^{\frac{1}{d+1}}\geq \vol_{\chi}(\bar L)^{\frac{1}{d+1}}+\vol_{\chi}(\bar M)^{\frac{1}{d+1}}.
$$
If $\bar L$ and $\bar M$ are ample, then ``$\vol_{\chi}=\vol$'' in the above inequality and the result is the same as that in \cite{Yu2}.

\section{Chebyshev Transform}
In this section, we prove Proposition \ref{definition} and Theorem \ref{integration}.
Proposition \ref{definition} will be proved immediately in \S 2.2. The proof of Theorem \ref{integration} is given in 
\S 2.4 after some more preparations in \S 2.3.

\subsection{Conventions and prelimilary results}\label{preliminary}

\subsubsection*{Lexicographic order}

In this paper we denote $\NN=\{0,1,2,\cdots\}$. We endow $\NN^d$ with the usual lexicographic order. Namely, 
$(\alpha_1, \cdots, \alpha_d) < (\alpha'_1, \cdots, \alpha'_d)$ for two elements in $\NN^d$ if there is an $i\in \{1,\cdots, d\}$ such that
$(\alpha_1, \cdots, \alpha_{i-1})=(\alpha'_1, \cdots, \alpha'_{i-1})$ and $\alpha_i <\alpha_i'$.
The order is preserved by addition in the sense that $\alpha+\beta \leq \alpha'+\beta'$ if $\alpha \leq \alpha'$ and $\beta \leq \beta'$.

\subsubsection*{Adelic metrics}

We recall the notion of adelic metrized line bundle introduced by Zhang \cite{Zh}. 
In the end we impose one extra condition on the rationality of the value of the supremum norm at each place.

Let $X$ be a projective variety over a number field $K$, and $L$ be a line bundle over $X$. For any place $v$ of $K$, by a $v$-adic metric on $L$ we mean the assignment of a Galois invariant $\CC_v$-norm to the fibre $L_{\CC_v}(x)$ at any point $x\in X(\CC_v)$. Here $\CC_v$ denotes the completion of the algebraic closure of $K_v$. 

Let $\cx$ be an arithmetic variety and $\lbb=(\lb, \|\cdot\|)$ be a hermitian line bundle over $\cx$.
Then the generic fibre $\lb_K$ is a line bundle on $\cx_K$. 
For each place place $v$, we will have a natural $v$-adic metric $\|\cdot\|_v$ on $\lb_K$. 
If $v$ is archimedean, the metric is exactly the hermitian metric. 
If $v$ is non-archimedean, a point $x\in \cx(\CC_v)$ extends to $\tilde x: \mathrm{Spec}(O_{\CC_v})\rightarrow \cx_{O_{\CC_v}}$. Then $\tilde x^* \lb_{O_{\CC_v}}$
gives a lattice in $\lb_{\CC_v}(x)$, and thus induces a $\CC_v$-norm on $\lb_{\CC_v}(x)$. It gives $\|\cdot\|_v$. We call a collection $\{\|\cdot\|_v\}_v$ of $v$-adic metrics on $\lb_K$ over all places $v$ of $K$ obtained by this way an \emph{algebraic adelic metric} on $\lb_K$.

In general, let $(X, L)$ be as above. A collection $\{\|\cdot\|_v\}_v$ of $v$-adic metric on $L$ over all places $v$ of $K$ is called an \emph{adelic metric} if there is an algebraic adelic metric $\{\|\cdot\|'_v\}_v$ on $L$ satisfying the following coherence and continuity conditions:
\begin{itemize}
\item There exists a finite set $S$ of places of $K$ such that $\|\cdot\|_v=\|\cdot\|'_v$ for all $v\notin S$;
\item The quotient $\|\cdot\|_v/\|\cdot\|'_v$ is a continuous function on $X(\CC_v)$ for all places $v$.
\end{itemize}
We usually write $\bar L=(L, \{\|\cdot\|_v\}_v)$ and call it an \emph{adelic metrized line bundle}. 
It induces the supremum norms on $H^0(X_{\CC_v}, L_{\CC_v})$ given by 
$$
\|s\|_{v,\sup}=\sup_{x\in X(\CC_v)}  \|s(x)\|_v.
$$

In this paper, we assume further that the image of $\|\cdot\|_{v,\sup}: H^0(X_{K_v}, mL_{K_v})\to \RR$
is contained in the image of $|\cdot|_{v}: K_v\to \RR$ for any positive integer $m$ and any place $v$.
Algebraic adelic metrics automatically satisfy this condition. We assume this condition for \emph{adelic metrized line bundle} throughout this paper. 
Similar results may also hold for metrics without this condition if we replace $K_v$ by $\CC_v$ in the local considerations of this paper.

\subsubsection*{Approximation of convex cone}
We recall a theorem in Khovanskii \cite{Kh} and some consequences which will be used later. Similar results are used in \cite{LM} and \cite{Ny}.

Let $\Gamma$ be a sub-semigroup of $\NN^{d+1}$, and assume that $\Gamma$ generates $\ZZ^{d+1}$ as a group. Denote 
\begin{eqnarray*}
\Sigma(\Gamma)&=& \mathrm{\ closed\ convex\ cone\ of\ } \Gamma 
\mathrm{\ in\ } \RR^{d+1};\\
\Delta(\Gamma)&=& \Sigma(\Gamma)\cap (\RR^d\times \{1\});\\
\Gamma_m&=& \Gamma\cap (\NN^d\times \{m\}), \quad m\in \NN;\\
\Lambda_m(\Gamma)&=& \frac 1m \Gamma_m \subset \Sigma(\Gamma), \quad m\in \NN.
\end{eqnarray*}
We may also view $\Delta(\Gamma)$ as a subset of $\RR^d$, and $\Lambda_m(\Gamma)$ as a subset of $\QQ^d$. 
Then $\Delta(\Gamma)$ is a convex body of $\RR^d$ in the sense that it is convex and closed. 
The following result is a rephrasal of Proposition 3 in \S 3 of Khovanskii \cite{Kh}:
\begin{thm}[Khovanskii] \label{Khovanskii}
If $\Gamma$ is finitely generated as a semigroup, then there exists an element $\gamma\in \Gamma$ such that
$(\gamma+\Sigma(\Gamma))\cap \NN^{d+1} \subset \Gamma.$
\end{thm}

\begin{cor}\label{convex geometry}
For any convex body $D$ contained in $\Delta(\Gamma)^\circ$, there exists $m_0>0$ such that 
$$D\cap \Lambda_m(\Gamma)=D\cap \frac 1m \NN^d$$
for all integers $m>m_0$.
\end{cor}
\begin{proof}
It is immediately true if $\Gamma$ is finitely generated by the above theorem.
If $\Gamma$ is arbitrary, we can find a finitely generated sub-semigroup $\Gamma'\subset \Gamma$ such that 
$\Delta(\Gamma')^\circ$ contains $D$. Then apply the theorem on $\Gamma'$.
\end{proof}

For applications to $(X,\bar L)$ as in the introduction, we will take
$$\Gamma= \bigcup_{m\geq 1} \nu(H^0(X,mL))\times \{m\}.$$
The related notations are translated as: 
\begin{eqnarray*}
\Delta(\Gamma)&=&\Delta(L) \mathrm{\ Okounkov\ body};\\
\Gamma_m&=& \nu(H^0(X,mL))\times \{m\};\\
\Lambda_m(\Gamma)&=& \Lambda_m(L)\times \{1\}.
\end{eqnarray*}
By \cite[Lemma 2.2]{LM}, it generates the group $\ZZ^{d+1}$ if $L$ is big. Hence we can use the results above.

\subsection{Basic properties}
Resume the notations in the introduction. That is, $X$ is a projective variety over $K$ of dimension $d$, and 
$\bar L=(L, \{\|\cdot\|_v\}_v)$ is an adelic metrized line bundle on $X$ with $L$ big. As in the introduction, take a rational regular point $x_0\in X(K)$ which exists by enlarging $K$, take a local coordinate $t=(t_1,\cdots, t_d)$ at $x_0$, and a base local section $s_0$ that induces a trivialization of $L$. 
Then we have a valuation map $\nu=\nu_{t}$ on the global sections, 
and the Okounkov body $\Delta(L)=\Delta_{t}(L)$ is the closure of 
$$\Lambda(L)=\bigcup_{m\geq 1}\frac{1}{m} \nu(H^0(X,mL))$$
in $\RR^d$. 
Denote 
$$\Lambda_m(L)=\frac{1}{m} \nu(H^0(X,mL)).$$
It is a finite subset of $\Lambda(L)$.

We first show Proposition \ref{definition} following the method of \cite{Ny}. 
Let $v$ be any place of $K$. For any positive integer $m$ and $\alpha\in \Lambda_m(L)$, we would like to give a lower bound of
$$
F_v[m\bar L](m\alpha)= \inf_{s\in H^0(X_{K_v},mL_{K_v})(m\alpha)} \{  \log \|s\|_{v,\sup} \},
$$
where
$$H^0(X_{K_v},mL_{K_v})(m\alpha)= \{s \in H^0(X_{K_v},mL_{K_v}): 
s=t^{m\alpha}+\mathrm{\ higher\ order\ terms} \}.$$

Denote by $|(m\alpha,m)|$ the sum of all components of $m\alpha$ and $m$.
The following result includes the non-archimedean case of \cite[Lemma 5.4]{Ny}. 

\begin{lem}\label{lower bound}
There exists a constant $C$ independent of $(m, \alpha)$ such that 
$$
F_v[m\bar L](m\alpha)\geq  C|(m\alpha,m)|.
$$
\end{lem}

\begin{proof}
The archimedean case is just Nystr\"om's result. 
The non-archimedean case is proved similarly.
Let $v$ be non-archimedean.
We will show that there is a constant $C$ such that 
$$\log \|s\|_{v,\sup} \geq C|(m\alpha,m)|, \quad \forall s\in H^0(X_{K_v},mL_{K_v})(m\alpha).$$
The coordinate $t=(t_1, \cdots, t_d)$ maps points near $x_0$ in 
$X(\CC_v)$ to $\CC_v^d$. 
We can find an closed neighborhood $U$ of $x_0$ which is 
bijective to the closed polydisc $D$ of radius $r>0$ in $\CC_v^d$.
As in the archimedean case, we have 
$$
\|s(x)\|_v= h(x)^m |s_0^{-m}s|_v, \quad \forall x\in U.
$$
Here $h$ is some positively-valued continuous function on $U$. 
Let $A>0$ be a positive lower bound of $h$ on $U$.
It follows that 
$$
\|s\|_{v,\sup}\geq \sup_{x\in U} \{ h(x)^m |s_0^{-m}s|_v\}
\geq A^m \sup_{x\in U} |s_0^{-m}s|_v.
$$
View $s_0^{-m}s$ as a convergent power series on $D$. It is of the form: 
$$t^{m\alpha}+ \mathrm{\ higher\ order\ terms} .$$
Its maximal absolute value on $D$ is equal to its Gauss norm. By definition, the Gauss norm is
greater than or equal to $r^{|m\alpha|}$.   
It finishes the proof.
\end{proof}

With the above lemma, Proposition \ref{definition} is proved in the same way as \cite[Proposition 5.6]{Ny}. We omit the details here.
Now we want to have some bound on $c_v[\bar L]$ by varying $v$. 
In the following lemma, ``almost all places'' means ``all but finitely many places''.

\begin{lem} \label{trivial bounds}
The following are true:
\begin{itemize}
\item[(1a)] For all places $v$, $\displaystyle \frac{1}{m} F_v[m\bar L](m\alpha)$ has a lower bound independent of $(m,\alpha)$. 
\item[(1b)] For all places $v$, $c_v[\bar L]$ is bounded from below.
\item[(2a)] For almost all places $v$, we have $F_v[m\bar L] \geq 0$ for all $m\geq 1$.
\item[(2b)] For almost all places $v$, and $c_v[\bar L]\geq 0$.
\item[(3a)] For any $\alpha \in \Lambda(L)$, there is a finite set $S$ of places of $F$ such that
$F_v[m\bar L](m\alpha)=0$ for all places $v\notin S$ and all positive integer $m$ such that $\alpha\in \Lambda_m(L)$.
\item[(3b)] For any $\alpha \in \Delta(L)^\circ$, the value $c_v[\bar L](\alpha)=0$ 
for all but finitely many places $v$.
\end{itemize}
\end{lem}

\begin{proof}
It is easy to see that (1b) and (2b) are implied by (1a) and (2a) respectively. From (3a) to (3b) it requires a little argument.

By Lemma \ref{lower bound}, 
$$ \frac{1}{m} F_v[m\bar L](m\alpha)> C|(\alpha, 1)|, \quad \alpha\in\Lambda_m(L).$$
The right-hand side is bounded since the Okounkov body is bounded. It proves (1a)

Now we prove (2a) and (3a). It will not impact the truth of the results if we change the adelic metric of $\bar L$ at finitely many places. Thus we can assume that the adelic metric on $\bar L$ is algebraic, i.e., it is induced by an integral model $(\cx, \lbb)$ of $(X,L)$. 

We first show (2a).
Let $v$ be a non-archimedean place such that the fibre $\cx_{\FF_v}$ of $\cx$ above $v$ is irreducible and has multiplicity zero in 
$\divv(s_0)$ and all $\divv(t_i)$ as divisors on $\cx$. It only excludes finitely many places. We claim that 
$F_v[m\bar L] \geq 0$. 
Otherwise, there is a section $s\in H^0(X_{K_v},mL_{K_v})(m\alpha)$ with $\|s\|_{v,\sup}<1$ for some $\alpha\in \Lambda_m(L)$. 
Then $\divv(s)$, as a divisor on $\cx$, has positive multiplicity on $\cx_{\FF_v}$. 
This is impossible since 
$$s = s_0^{m}t^{m\alpha}(1+ \mathrm{\ higher\ order\ terms} ).$$

Now we consider (3a).
By (2a), it suffices to show that $F_v[m\bar L](m\alpha)\leq 0$ for almost all places $v$ and for all $m\in M(\alpha)$. Here $M(\alpha)$ denote the set of positive integer $m$ such that $\alpha\in \Lambda_m(L)$. It is a semigroup.
Fix an $m\in M(\alpha)$, and pick any element $s\in H^0(X_{K_v},mL_{K_v})(m\alpha)$.
Then $\|s\|_{v,\sup}=1$ for almost all non-archimedean place $v$.
For such a place $v$, by definition $F_v[m\bar L](m\alpha)\leq 0$. 
In another word, there is a finite set $S(m)$ of places of $K$ such that $F_v[m\bar L](m\alpha)\leq 0$ for all 
$v\notin S(m)$.
Observe that for $m_1, m_2\in M(\alpha)$, we have $m_1+m_2\in M(\alpha)$ and $S(m_1+m_2)\subset S(m_1)\cup S(m_2)$. 
Since the semigroup $M(\alpha)$ is finitely generated, it is easy to find a common $S$ for all $m\in M(\alpha)$.

In the end, we consider (3b).
By (3a), we see that the result is true if $\alpha\in \Lambda(L)$.
For a general $\alpha \in \Delta(\bar L)^\circ$, pick two elements $\beta, \gamma \in \Lambda(\bar L)$ such that the line segment between $\beta$ and $\gamma$ contains $\alpha$. We have showed $c_v[\bar L](\beta)= c_v[\bar L](\gamma)= 0$ for almost all $v$.
Since $c_v[\bar L]$ is convex, we have $c_v[\bar L](\alpha)\leq 0$ almost all $v$. By (1b), we conclude that 
$c_v[\bar L](\alpha)= 0$ for almost all $v$.
\end{proof}

\begin{remark}
By the convexity of $c_v[\bar L]$, we actually know that for any closed convex polytope contained in $\Delta(\bar L)^\circ$, $c_v[\bar L]$ vanishes identically on the polytope for almost all $v$.
\end{remark}

Now we are can introduce the global Chebyshev transform as the sum of the local ones.
Denote 
\begin{eqnarray*}
F[\bar L](\alpha):&= &\sum_{v} F_v[\bar L](\alpha), \quad \alpha\in \Lambda_1(L)\\
c[\bar L](\alpha):&= &\sum_{v} c_v[\bar L](\alpha), \quad \alpha\in \bodyin.
\end{eqnarray*}
By Lemma \ref{trivial bounds} (3a) (3b), both sums are finite and thus well-defined. 
Below is the global version of Proposition \ref{definition}.

\begin{pro}\label{global limit}
The following are true:
\begin{itemize}
\item[(1)] The function $c[\bar L]$ is convex and continuous on $\bodyin$. 
\item[(2)] Let $\{m_k\}_k$ be a sequence of positive integers convergent to infinity, 
and $\alpha_k\in \Lambda_{m_k}(L)$ be a sequence convergent to some $\alpha\in \Delta(L)^\circ$. 
Then
$$
c[\bar L](\alpha)= \lim_{k\rightarrow\infty} \frac{1}{m_k} F[m_k\bar L](m_k\alpha_k).
$$
\end{itemize}
\end{pro}

\begin{proof}
These properties can be transfered from the local case since the definitions of $F[\bar L]$ and $c[\bar L]$ are essentially summations. For example, (1) is immediate. 

Now we prove (2). We claim that there is
a finite set $S$ of places such that for any $v\notin S$,
$$F_v[m_k\bar L](m_k\alpha_k)= 0.$$
Once this is true, the result follows from the local case.

By Lemma \ref{trivial bounds} (2a), there is a 
a finite set $S$ of places of $F$ such that for any $v\notin S$,
$$F_v[m_k\bar L](m_k\alpha_k)\geq 0.$$
It remains to show that the inverse direction of the inequality is true for some $S$. 

As in \S \ref{preliminary}, denote
$$\Gamma= \bigcup_{m\geq 1} \nu(H^0(X,mL))\times \{m\}.$$
By definition $(m_k\alpha_k,m_k)$ lies in the closed convex cone $\Sigma(\Gamma)$.
Since $\alpha_k\rightarrow \alpha$, we can find a sub-semigroup $\Gamma'$
generated by finitely many points
$(n_1\beta_1,n_1), \cdots, (n_r\beta_r,n_r)$ of $\Gamma$ such that
$\Sigma(\Gamma')$ contains $(m_k\alpha_k,m_k)$ for all $k$.
By Theorem \ref{Khovanskii}, $(m_k\alpha_k,m_k)\in \Gamma'$ for $k$ large enough.

By Lemma \ref{trivial bounds} (3a), there is a 
a finite set $S$ of places of $F$ such that for any $v\notin S$,
$$F_v[n_j\bar L](n_j\beta_j)=0, \quad j=1,\cdots, r.$$
It implies that for any $v\notin S$,
$$F_v[m_k\bar L](m_k\alpha_k)\leq 0, \quad\forall k\gg0.$$
In fact, for sufficiently large $k$ one can write
$$(m_k\alpha_k,m_k)=\sum_j a_j (n_j\beta_j,n_j), \quad a_j\in \NN.$$
Then any choice of section $s_j\in H^0(X_{K_v},n_j L_{K_v})(n_j\beta_j)$
gives a section 
$$\otimes_j s_j^{\otimes a_j} \in H^0(X_{K_v},m_k L_{K_v})(m_k\alpha_k).$$
Then it is easy to have the bound.

\end{proof}

\subsection{Euler characteristic and arithmetic degree}

Let $X, \bar L$ be as before.
In this subsection we build relation between $F[\bar L]$ and $\chi(\bar L)$. 
The main result of this subsection is

\begin{thm}\label{summation}
\begin{eqnarray*}
\chi(m \bar L) = - \sum_{\alpha\in \Lambda_m(L)} F[m\bar L](m\alpha)+O(m^d\log m).
\end{eqnarray*}
\end{thm}

\subsubsection*{Change of norms}
The supremum norm at archimedean place is not an inner product, so we introduce the $L^2$-norm and compare these two norms.

Let $v$ be an archimedean place. 
Fix a measure $d\mu_v$ on $X(\CC_v)$, which is assumed to be the push-forward measure of a positive smooth volume form on some resolution of singularity of $X(\CC_v)$. It gives an $L^2$-norm by 
$$\|s\|_{v,L^2}^2= \int \|s(z)\|_v^2 d\mu_v , \quad s \in H^0(X, L)_{\CC_v}.$$
Denote the associated bilinear pairing by $\langle \cdot, \cdot \rangle_v$.

\begin{lem}[Gromov]\label{Gromov}
There exists a positive constant $a$ such that
$$am^{-d} \|s\|_{v,\sup} \leq \|s\|_{v,L^2} \leq \|s\|_{v,\sup} \ ,  \ \forall m>0, s\in H^0(X, m L)_{\CC_v}. $$
\end{lem}

For a proof, see \cite{GS} or \cite[Proposition 2.13]{Yu1}. 
Such a property is called Bernstein-Markov property by analysists.

Replacing the supremum norms by the $L^2$-norms at every archimedean place, we define the $L^2$-version $(\chil(\bar L), F'_v[\bar L],c'_v[\bar L])$ of $(\chi(\bar L), F_v[\bar L],c_v[\bar L])$.
Their difference can be ignored by the above lemma.

First, the $L^2$-characteristic $\chil$ is defined by
\begin{eqnarray*}
\chil(\bar L) = \log \frac{\vol(B_{L^2}(\bar L))}{\vol(H^0(X,L)_{\AA_K}/H^0(X,L))}.
\end{eqnarray*}
Here the unit ball
$$
B_{L^2}(\bar L)=\prod_{v|\infty } B_{v,L^2}(\bar L)\times \prod_{v\nmid\infty } B_v(\bar L)
$$
with
$$B_{v,L^2}(\bar L)=\{s\in H^0(X,L)_{K_v}: \|s\|_{v,L^2}\leq 1 \}.$$

Second, for any archimedean $v$, we define
$$
F'_v[m\bar L](m\alpha):= \inf_{s\in H^0(X_{K_v},mL_{K_v})(m\alpha)} \{  \log \|s\|_{v,L^2} \}.
$$
For any $\alpha\in \Delta(L)^\circ$, choose a sequence $\alpha_k\in \Delta_{m_k}(L)$ converging to a point $\alpha$,
and define 
$$
c'_v[\bar L](\alpha):= \lim_{k\rightarrow\infty} \frac{1}{m_k} F_v'[m_k\bar L](m_k\alpha_k).
$$
We have the following simple consequence of Lemma \ref{Gromov}.

\begin{lem}\label{norm comparison}
Let $v$ be an archimedean place. The following are true:
\begin{itemize}
\item[(1)]  $\chil(m\bar L)= \chi(m\bar L)+O(m^d\log m).$
\item[(2)]  $F'_v[m\bar L](m\alpha)= F_v[m\bar L](m\alpha)+O(\log m).$
\item[(3)]  For any $\alpha\in\Delta(L)^\circ$, the limit defining $c'_v[\bar L](\alpha)$ exists and does not depend on the sequence $\{\alpha_k\}$. Furthermore, $c'_v[\bar L]=c_v[\bar L]$. 
\end{itemize}
\end{lem}

\subsubsection*{Euler characteristic and arithmetic degree}
Let $s_1, s_2, \cdots, s_N$ be a basis of $H^0(X,L)$ over $K$, where $N=\dim_K H^0(X,L)$.
Define
\begin{eqnarray*}
\deg H^0(X,\bar L) := - \sum_{v\mid\infty} \log \sqrt{\det(\langle s_i,s_j\rangle_v)_{i,j}}
-\sum_{v\nmid \infty }\log \frac{\vol(O_{K_v} s_1+ \cdots+O_{K_v} s_N)}{\vol(B_v(\bar L))}.
\end{eqnarray*}

\begin{pro}\label{degree}
The definition of $\deg H^0(X,\bar L)$ is independent of the basis $s_1, s_2, \cdots, s_N$.
Furthermore, there exists a positive constant $a_K$ depending only on $K$ such that 
\begin{eqnarray*}
|\chil(X,\bar L) - \deg H^0(X,\bar L)|\leq a_KN\log N.
\end{eqnarray*}
\end{pro}

\begin{proof}
We first recall some basic algebraic number theory. One has a coset identity 
$$\AA_K/K \cong \left(K\otimes_{\QQ}\RR /O_K\right) \times \prod_{v\nmid \infty} O_{K_v}.$$
Here $K\otimes_{\QQ}\RR = \RR^{r_1}\times \CC^{r_2}$ where $r_1$ (resp. $2r_2$) is the number of real (resp. imaginary) embeddings of $K$ in $\CC$. Endow $K\otimes_{\QQ}\RR$ with the Lebesgue measure induced by the identity, then
$ \vol(K\otimes_{\QQ}\RR /O_K)=\sqrt{d_K}$. Here $d_K$ denotes the discriminant of $K$ over $\QQ$.

Go back to the current situation. By the product formula, it is easy to see that the definition is independent of the basis. Denote by $M=O_K s_1+ \cdots+O_K s_N$
the free lattice in $H^0(X,L)$ generated by the basis.
We obtain 
$$H^0(X,L)_{\AA_K}/H^0(X,L)\cong 
\left(H^0(X,L)\otimes_{\QQ}\RR /M\right) \times \prod_{v\nmid \infty} M_{O_{K_v}}.$$
It follows that 
\begin{eqnarray*}
\chil(X,\bar L) =  \log \frac{\vol(\prod_{v|\infty } B_{v,L^2}(\bar L))}{ \vol(H^0(X,L)\otimes_{\QQ}\RR /M) }
+\sum_{v\nmid \infty }\log \frac{\vol(B_v(\bar L))}{\vol(M_{O_{K_v}})}.
\end{eqnarray*}
Identify $H^0(X,L)$ with $K^N$ via the basis $\{s_i\}$.
Endow $H^0(X,L)\otimes_{\QQ}\RR \cong \RR^{r_1N}\times \CC^{r_2N}$ with the Lebesgue measure.
We have $\vol(H^0(X,L)\otimes_{\QQ}\RR /M)=\sqrt{d_K^N}$, and 
$$
\vol(B_{v,L^2}(\bar L))= \begin{cases}
\det(\langle s_i,s_j\rangle_v)^{-\frac 12} V(N)  & \mathrm{if\ } v \mathrm{\ is \ real};\\
\det(\langle s_i,s_j\rangle_v) V(2N)  & \mathrm{if\ } v \mathrm{\ is \ imaginary}.
\end{cases}
$$
Here $$\displaystyle V(N)=\pi^{\frac{N}{2}}/\Gamma(\frac{N}{2}+1)$$
is the volume of the unit ball in the Euclidean space $\RR^N$. 
It follows that
\begin{eqnarray*}
\chil(X,\bar L) = \deg H^0(X,\bar L)-\frac 12 N\log d_K +r_1\log V(N)+r_2\log V(2N).
\end{eqnarray*}
Then the result follows from Stirling's formula.

\end{proof}

\begin{remark}
The vector space $H^0(X,L)$ over $K$ is endowed with the $L^2$-norm at archimedean places and the supremum norms at non-archimedean places. Then it induces an adelic metric on the one-dimensional $K$-vector space $\det H^0(X,L)$, whose arithmetic degree is exactly $\deg H^0(X,\bar L)$ defined above. A more elegant expression is 
\begin{eqnarray*}
\deg H^0(X,\bar L) = - \sum_{v} \log\| s_1\wedge s_2\wedge \cdots\wedge s_N \|_v.
\end{eqnarray*}
The bound above is a Riemann-Roch type result.
\end{remark}

The following is the fundamental identity between the Euler characteristic and Chebyshev transform. 
For convenience, for non-archimedean $v$, we set $F'_v:=F_v$.

\begin{pro}\label{fundamental identity}
\begin{eqnarray*}
\deg H^0(X,\bar L) = -\sum_{v} \sum_{\alpha\in \Lambda_1(L)} F'_v[\bar L](\alpha).
\end{eqnarray*}
\end{pro}

\begin{proof}
For every $\alpha\in \Lambda_1(L)$, we pick an element $s_{\alpha}$ of
$$H^0(X,L)(\alpha)= \{s \in H^0(X,L): \
s = t^{\alpha}+ \mathrm{\ higher\ order\ terms} \}.$$
Then $\{s_{\alpha}: \alpha\in \Lambda_1(L) \}$ forms a $K$-basis for $H^0(X,L)$.
By definition, 
\begin{eqnarray*}
\deg H^0(X,\bar L) = - \sum_{v\mid\infty} 
\log \sqrt{\det(\langle s_\alpha,s_\beta\rangle_v)_{\alpha,\beta\in \Lambda_1(L)}}
-\sum_{v\nmid \infty }\log \frac{\vol(\sum_{\alpha\in \Lambda_1(L)} O_{K_v} s_\alpha)}{\vol(B_v(\bar L))}.
\end{eqnarray*}
We will show the match of the local terms at each place $v$.

Recall that for $\alpha\in \Lambda_1(L)$,
$$
F'_v[\bar L](\alpha)=\inf_{s\in H^0(X_{K_v},L_{K_v})(\alpha)} \{  \log \|s\|_{v,\sup} \}.
$$
Let $e_{v,\alpha}\in H^0(X_{K_v},L_{K_v})(\alpha)$ be an element that takes the infimum.
Then we have 
$$
\sum_{\alpha\in \Lambda_1(L)} F'_v[\bar L](\alpha)
=\log \prod_{\alpha\in \Lambda_1(L)} \|e_{v,\alpha}\|_{v,\sup}
$$
for non-archimedean $v$. 
One needs to replace the supremum norm by the $L^2$-norm in the above expression for archimedean $v$.

First assume that $v$ is archimedean.
Since both $s_\alpha$ and $e_{v,\alpha}$ are elements of $H^0(X_{K_v},L_{K_v})(\alpha)$,
the transition matrix between the basis 
$\{e_{v,\alpha}\}_\alpha$ and $\{s_{\alpha}\}_\alpha$ is upper-triangular with 1 on the diagonals.
Thus it has determinant 1.
It follows that 
$$\det(\langle s_\alpha,s_\beta\rangle_v)_{\alpha,\beta\in \Lambda_1(L)}=
\det(\langle e_{v,\alpha},e_{v,\beta}\rangle_v)_{\alpha,\beta\in \Lambda_1(L)}.$$

A key property of $\{e_{v,\alpha}\}_\alpha$ is that they form an orthogonal basis of $H^0(X_{\CC_v},L_{\CC_v})$.
Otherwise, assume that $e_{v,\alpha}$ is not orthogonal to $e_{v,\alpha'}$
for some $\alpha<\alpha'$. Then there will be an $\epsilon\in K_v$ such that the norm of
$e_{v,\alpha}+ \epsilon e_{v,\alpha'}$ is greater than 
$e_{v,\alpha}$. It contradicts to the choice of $e_{v,\alpha}$.

Hence, we simply have 
$$\det(\langle s_\alpha,s_\beta\rangle_v)_{\alpha,\beta\in \Lambda_1(L)}=
\prod_{\alpha\in \Lambda_1(L)}
\|e_{v,\alpha}\|_{v,L^2}^2.$$
It gives the matching
$$
\log \sqrt{\det(\langle s_\alpha,s_\beta\rangle_v)_{\alpha,\beta\in \Lambda_1(L)}}=
\sum_{\alpha\in \Lambda_1(L)} F'_v[\bar L](\alpha).$$

Next, assume that $v$ is non-archimedean.
The transition matrix between $s_\alpha$ and $e_{v,\alpha}$ still has determinant 1.
It follows that we have 
$$
\vol(\sum_{\alpha\in \Lambda_1(L)} O_{K_v} s_\alpha)=\vol(\sum_{\alpha\in \Lambda_1(L)} O_{K_v} e_{v,\alpha}).
$$
It remains to show that 
$$
\frac{\vol(\sum_{\alpha\in \Lambda_1(L)} O_{K_v} e_{v,\alpha})}{\vol(B_v(\bar L))}
=\prod_{\alpha\in \Lambda_1(L)} \|e_{v,\alpha}\|_{v,\sup}.
$$
We claim that 
$$
B_v(\bar L)=\sum_{\alpha\in \Lambda_1(L)} O_{K_v} e_{v,\alpha}^{\circ},
$$
Here $e_{v,\alpha}^{\circ}$ is any element in $K_v e_{v,\alpha}$ with norm 1.
It is easy to see that it implies what we want.

Now we prove the claim.
Let $s\in B_v(\bar L)$ be any element. We need to show that 
$s$ belongs to $\sum_{\alpha} O_{K_v} e_{v,\alpha}^{\circ}$.
We can uniquely write
$$
s=\sum_{\alpha\in \Lambda_1(L)} a_{\alpha} e_{v,\alpha},\quad  a_{\alpha}\in K_v.
$$
It suffices to show that $|a_{\alpha}|_v \cdot  \|e_{v,\alpha}\|_{v,\sup}\leq 1$ for all $\alpha$.

Let $\alpha_1 < \alpha_2 < \cdots < \alpha_N$ be the elements of $\Lambda_1(L)$ in the lexicographic order.
If $a_{\alpha_1}\neq 0$, we have 
$$
1\geq \|s\|_{v,\sup}
\geq  |a_{\alpha_1}|_v \cdot 
\left\| e_{v,\alpha_1}+ \sum_{\alpha\neq \alpha_1} a_{\alpha_1}^{-1}a_{\alpha} e_{v,\alpha} \right\|_{v,\sup} 
\geq |a_{\alpha_1}|_v \cdot  \|e_{v,\alpha_1}\|_{v,\sup}.
$$
Here the last inequality follows from the minimality of $e_{v,\alpha_1}$.
It confirms the case $\alpha=\alpha_1$.
Considering 
$$
s-a_{\alpha_1} e_{v,\alpha_1}=\sum_{\alpha\neq \alpha_1} a_{\alpha} e_{v,\alpha}  \in B_v(\bar L),
$$
we can show that $|a_{\alpha_2}|_v \cdot \|e_{v,\alpha_2}\|_{v,\sup} \leq 1$.
Inductively, we will have the result for all $\alpha$.

\end{proof}

Now it is easy to prove Theorem \ref{summation}.
By Proposition \ref{degree} and Proposition \ref{fundamental identity}, 
\begin{eqnarray*}
\chil(m \bar L) = -\sum_{v} \sum_{\alpha\in \Lambda_m(L)} F'_v[m\bar L](m\alpha)+O(m^d\log m).
\end{eqnarray*}
By Lemma \ref{norm comparison}, we have the supremum version
\begin{eqnarray*}
\chi(m \bar L) = -\sum_{v} \sum_{\alpha\in \Lambda_m(L)} F_v[m\bar L](m\alpha)+O(m^d\log m).
\end{eqnarray*}
By Lemma \ref{trivial bounds} (3a), there are only finitely many nonzero terms in the double summation, so we can change the order of the summation. It gives the identity in the theorem.

\subsection{Proof of the main result}
Now we are ready to prove Theorem \ref{integration}.
Note that if we change the metric $\|\cdot\|_v$ to $e^{-a}\|\cdot\|_v$ for some constant $a\in\RR$, then
$F_v[\bar L]$ will increase by $a$. By this fact it is easy to see that the truth of Theorem \ref{integration} does not change. By Lemma \ref{trivial bounds}, we can assume that all $F_v[m\bar L]\geq 0$ for all $v$ and all $m$.
Then everything involved in the summation and integrals below are non-negative.

\subsubsection*{The first part}
By Theorem \ref{summation},
\begin{eqnarray}\label{111}
\frac{\chi(m\bar L)}{m^{d+1}} = 
-\frac{1}{m^d} \sum_{\alpha\in \Lambda_m(L)} \frac{1}{m}F[m\bar L](m\alpha)+O(\frac 1m\log m).
\end{eqnarray}

Observe that 
$$\frac{1}{m}F[m\bar L](m\alpha)\geq c[\bar L](\alpha).$$
In fact, for any $\alpha\in \Lambda_m(L)$, the sequence $\displaystyle\frac{1}{mk}F_v[mk\bar L](mk\alpha)$ decreases to $c_v[\bar L](\alpha)$ as $k\rightarrow\infty$. Hence the limit is always smaller.
It follows that (\ref{111}) gives
\begin{eqnarray*}
-\frac{\chi(m\bar L)}{m^{d+1}} \geq  
 \frac{1}{m^d} \sum_{\alpha\in \Lambda_m(L)} c[\bar L](\alpha)+O(\frac 1m\log m).
\end{eqnarray*}

By Corollary \ref{convex geometry}, for any convex body $D$ contained in $\Delta(L)^\circ$, 
\begin{eqnarray*}
-\frac{\chi(m\bar L)}{m^{d+1}} \geq  
 \frac{1}{m^d} \sum_{\alpha\in D\cap \frac 1m \NN^d} c[\bar L](\alpha)+O(\frac 1m\log m).
\end{eqnarray*}
The right-hand side is a Riemann sum for $c[\bar L]$, except for some exceptions on the boundary of $D$ which can be ignored. Taking limit, we obtain
$$
-\frac{1}{d!}\vol_{\chi}(\bar L) \geq \int_{D} c[\bar L](\alpha)d\alpha.
$$
Let $D\rightarrow \Delta(L)$. We obtain
$$
-\frac{1}{d!}\vol_{\chi}(\bar L) \geq \int_{\Delta(L)} c[\bar L](\alpha)d\alpha.
$$
Then we see that $c[\bar L]$ and all $c_v[\bar L]$ are integrable by the following lemma.

\subsubsection*{The second part}
Assume that $\displaystyle \frac{1}{m}F[m\bar L]$ is uniformly bounded. Then our regularization method is the similar to Nystr\"om's proof. 

Recall that in (\ref{111}) we have
\begin{eqnarray*}
\frac{\chi(m\bar L)}{m^{d+1}} = 
-\frac{1}{m^d} \sum_{\alpha\in \Lambda_m(L)} \frac{1}{m}F[m\bar L](m\alpha)+O(\frac 1m\log m).
\end{eqnarray*}
Define a step function $\widetilde c_m: \RR^n\rightarrow \RR$ as follows. 
Define 
$$
\widetilde c_m(\alpha+ \beta)=\frac{1}{m}F[m\bar L](m\alpha), 
\quad \forall \alpha\in \Lambda_m(L),\  \beta \in [0, 1/m)^d.
$$
Define $\widetilde c_m(x)$ to be zero if $x\in \RR^n$ can not be written in the form $\alpha+ \beta$ describe above form. Then the summation on the right-hand side is exactly equal to the integral of $\widetilde c_m$.
We get  
\begin{eqnarray}\label{222}
\frac{\chi(m\bar L)}{m^{d+1}} = 
-\int_{\RR^n} \widetilde c_m(x)dx+O(\frac 1m\log m).
\end{eqnarray}
By Proposition \ref{global limit} and Corollary \ref{convex geometry}, we know that 
$\widetilde c_m(x)$ converges to $c[\bar L](x)1_{\bodyin}(x)$ almost everywhere. 
Note that $\widetilde c_m(x)$ is uniformly bounded and uniformly supported on a bounded domain. We can use dominant convergence theorem to conclude that 
\begin{eqnarray*}
\lim_{m\rightarrow\infty} \int_{\RR^n} \widetilde c_m(x)dx= \int_{\Delta(L)} c[\bar L](x)dx.
\end{eqnarray*}
Take limit in (\ref{222}). It proves the result.

\

We conjecture that $\displaystyle \frac{1}{m}F[m\bar L]$ is uniformly bounded for all big $L$. 
It is easy to see that the result does not depend on the adelic metric on $L$.
In fact, if $\{\|\cdot\|_v\}_v$ and $\{\|\cdot\|'_v\}_v$ are two metrics of $L$, 
then we can find a positive constant $b_v\geq 1$ for each place $v$ such that 
$b_v^{-1}\|\cdot\|_v\leq \|\cdot\|'_v \leq b_v\|\cdot\|_v$. Furthermore, we can take $b_v=1$ for almost all $v$.
Then  
$$ -\sum_v\log b_v
\leq\sum_{v} \frac{1}{m}F_v[\|\cdot\|'_v](m\alpha,m)-
\sum_{v} \frac{1}{m}F_v[\|\cdot\|_v](m\alpha,m)
\leq\sum_v\log b_v.$$
Then one of the sums is bounded if and only if the other one is bounded.

The following are some simple examples for which the uniform bound is easy to obtain.
\begin{example}
\noindent \textbf{(1)} The standard valuation in projective space.
More precisely, $X=\PP^d$ with homogeneous coordinate $(Z_0, \cdots, Z_d)$,
base point $x_0=(1,0,\cdots, 0)$, and local coordinate $t=(Z_1/Z_0, \cdots, Z_d/Z_0)$.
The line bundle $\bar L=(O(1), \{\|\cdot\|_v\}_v)$ with arbitrary adelic metric, and the trivialization is given by the base section $s_0=Z_0$.
Then it is easy to verify that
$$
\frac{1}{m}F[m\bar L](m\alpha) \leq \max_{0\leq i\leq d} \sum_{v}\log\|Z_i\|_{v,\sup}.
$$
Note that the summation on the right-hand side is always a finite sum by the definition of adelic metric. 
\end{example}

\textbf{(2)} Assume that the value semi-group 
$$\Gamma= \bigcup_{m\geq 1} \nu(H^0(X,mL))\times \{m\}$$
is finitely generated. 
As in the proof of Proposition \ref{global limit}, we have 
$$
\frac{1}{m}F[m\bar L](m\alpha) \leq 
\max_{1\leq j\leq r} \frac{1}{n_j}F[n_j\bar L](n_j\beta_j)
$$
where $\{(n_j\beta_j,n_j): j=1,\cdots,r\}$ is a set of generator of $\Gamma$.
It happens if $X$ is a toric variety and $x_0, \divv(s_0), \divv(t_1), \cdots, \divv(t_d)$ 
are invariant under the torus action.

\

\

{\footnotesize
Address: Department of Mathematics, Harvard University, One Oxford Street, Cambridge, MA 02138

Email: yxy@math.harvard.edu
}

\end{document}